\documentclass{amsart}

\usepackage{amsfonts,amsmath,amssymb}
\usepackage{bm}
\usepackage[OMLmathrm,OMLmathbf]{isomath}
\usepackage[english]{babel}
\usepackage[utf8]{inputenc}
\input xy
\xyoption{all}
\newtheorem{thm}{Theorem}

\usepackage{graphicx}
\usepackage{bm}
\usepackage{amssymb}

\date{}

\title{A note on Hutchinson operator in $\mathbf{T_1}$ compact topological spaces}

\address{Micha{\l} Morayne: Faculty of Fundamental Problems of Technology, Wroc{\l}aw University of Science and Technology, 50-370 Wrocław, Poland.}
\address{Robert Ra{\l}owski: Faculty of Pure and Applied Mathematics, Wrocław University of Science and Technology, 50-370 Wroc{\l}aw, Poland.}
\author{Micha{\l} Morayne}
\author{Robert Ra{\l}owski}
\thanks{The work of Robert Rałowski has been partially financed by  grants 8211204601, 8211104160, MPK: 9120730000 from the Faculty of Pure and Applied Mathematics, Wrocław University of Science and Technology}
\email{michal.morayne@pwr.edu.pl}
\email{robert.ralowski@pwr.edu.pl}


\begin{document}

\begin{abstract}
In $T_1$ compact topological spaces the Hutchinson operator of a contractive IFS 
(iterated function system; a finite family of closed mappings from the space into itself) may not be closed. Nevertheless, the Hutchinson operator of a contractive IFS has always a unique fixed point.
\end{abstract}

\maketitle

 An {\it iterated function system} (for short: IFS) is any finite family  $\mathcal{F}=\{f_1,\ldots,f_m\}$ of closed mappings from $X$ to $X$. We say that an IFS $\mathcal{F}=\{f_1,\ldots,f_m\}$ is {\it contractive} if 
for any open covering $\mathcal{U}=\{U_\lambda:\lambda\in\Lambda\}$ of $X$ there exists $n\in\mathbb{N}$ such that for any sequence $(i_1,\ldots,i_n)$, $1\leq i_j\leq m$, the set $f_{i_1}\circ\ldots\circ f_{i_n}[X]$ is contained in some element $U_\alpha$ of the covering   $\mathcal{U}$. 
A mapping $f:X\to X$ is a {\it topological contraction} if the IFS consisting of this one mapping $f$ is contractive. Equivalently (Theorem 4 in \cite{MMRR}), $f:X\to X$ is a topological contraction if $f$ is closed and

\vspace{0.2 cm}

\noindent $(\star)$ for each two different points $x,y\in X$ there exists $n\in\mathbb{N}$ such that
$f^n[X]\subseteq X\setminus\{x\}$ or $f^n[X]\subseteq X\setminus\{y\}$.

\vspace{0.2 cm}

It follows from the Lebesgue number lemma that every IFS consisting of Lipschitz  contractions on a compact metric space  is contractive in the above sense. Let $2^X$ be the family of all closed nonempty 
subsets of $X$.  The Vietoris topology on $2^X$ is generated by the basis consisting of all sets of the form 
$$S(V_0;V_1,\ldots,V_k):=\{K\in2^X:K\subseteq V_0, K\cap U_i\neq\emptyset,1\leq i\leq k\},$$
where $k\in\mathbb{N}$ and $V_0,V_1,\ldots,V_k$ are open subsets of $X$. The space $2^X$ with the Vietoris topology is called the {\it hyperspace} of $X$. It is known that if $X$ is a $T_1$ compact space then so is its hyperspace $2^X$ (\cite{Mi}). 

Let $\mathcal{F}=\{f_1,\ldots,f_m\}$ be an IFS on a $T_1$ space $X$. The {\it Hutchinson operator $F: 2^X \to 2^X$ induced by $\mathcal{F}$} 
is defined by the formula 
$$F(K):=\bigcup_{i=1}^mf_{i}[K]$$
(see \cite{H}). If $K\in2^X$ is a fixed point of the Hutchinson operator induced by an IFS $\mathcal{F}$ we call $K$ an {\it atractor} of  $\mathcal{F}$.

In \cite{MMRR} the following theorems were proven (Theorems 3, 6, and Corollary 7 in \cite{MMRR}, resp.).

\vspace{0.4 cm}

\noindent{\bf Theorem A.} {\it
If $X$ is a $T_1$ compact space, $f:X\to X$ is a topological contraction, then 
there exists a unique fixed point for the mapping $f$.} 

\vspace{0.4 cm}

\noindent{\bf Theorem B.} {\it 
Let $X$ be a $T_1$ compact space. Let $\mathcal{F}$ be a contractive IFS on $X$. Then, if the Hutchinson operator $F$ induced by $\mathcal{F}$ is closed, then $F$ is a topological contraction on $2^X$.} 

\vspace{0.4 cm}

\noindent{\bf Theorem C.} {\it
If $X$ is a $T_1$ compact space then every contractive IFS inducing a closed Hutchinson operator has a unique attractor; in other words: the Hutchinson operator induced by this IFS has a unique fixed point.}

\vspace{0.4 cm}

Of course, Theorem C follows from Theorems A and B.

The attractor in the conclusion of Theorem C is unique. One should add that the very existence of an attractor for an IFS on $T_1$ compact space does not require any additional assumptions imposed on the IFS (for instance, that it is contractive). Namely the following theorem (Theorem 5 
in \cite{MMRR}) holds.

\vspace{0.4 cm}

\noindent{\bf Theorem D.} {\it
If $X$ is a $T_1$ compact space then any IFS on $X$ has an attractor; in other words the  Hutchinson operator $F$ induced by $\mathcal{F}$ has a fixed point.}  

\vspace{0.4 cm}

There arises a question whether the Hutchinson operator induced by a contractive IFS is always closed.
We  shall construct an example of a $T_1$ compact topological space and a contractive IFS on this space inducing a Hutchinson operator which is not closed. The existence of such a space and IFS we state as a theorem.

\begin{thm}\label{Hutchinson-example}
There exists a $T_1$ compact space $X$ and a contractive IFS inducing a Hutchinson operator  which is not closed (as a mapping from $2^X$ to $2^X$). 
\end{thm}

\begin{proof} Let 
$$\textrm{ODD}=\{1,3,5,\ldots\}$$
and
$$\textrm{EVEN}=\{2,4,6,\ldots\}.$$
Let 
$$X=\mathbb{N}\cup\{a,b\}$$
where $a\neq b$, $a,b\notin\mathbb{N}$, and the set $A\subseteq X$ is open if:

\vspace{0.2 cm}

\noindent i) $A\subset \textrm{ODD}$,

\vspace{0.2 cm}

or  

\vspace{0.2 cm}

\noindent ii) $A\subseteq\mathbb{N}$ and the set $\textrm{ODD}\setminus A$ is finite,

\vspace{0.2 cm}

or  

\vspace{0.2 cm}

\noindent iii) $A\cap\{a,b\}\neq\emptyset$ and $A$ is a co-finite set in $X$.

\vspace{0.2 cm}

We define an IFS $\mathcal{F}$ as 
$$\mathcal{F}=\{f,g\},$$
where 
$$f(x)=\left\{
\begin{array}{ccc}
a&\textrm{if}&x=a,b,\\
b&\textrm{if}&x=2n,n\in\mathbb{N},\\
2n&\textrm{if}&x=2n-1,n\in\mathbb{N},\\
\end{array} 
\right.
$$
and
$$g(x)=\left\{
\begin{array}{ccc}
b&\textrm{if}&x=a,b,\\
a&\textrm{if}&x=2n,n\in\mathbb{N},\\
2n&\textrm{if}&x=2n-1,n\in\mathbb{N}.
\\
\end{array} 
\right.
$$

\vspace{0.2 cm}

Let $E\subseteq X$ be a nonempty closed set. If $E$ is finite,
then $f[E]$ is also finite and, therefore, closed. 
If $E$ is infinite, then $a,b\in E$ and $E\cap\textrm{EVEN}\neq\emptyset$.
Then $\{a,b\}\subseteq f[E]\subseteq\{a,b\}\cup\textrm{EVEN}$ which is a closed set. 

Thus $f$ is a closed mapping. Analogously, one argues that $g$ is closed.

Now let us consider a composition $h_n\circ h_{n-1}\circ\ldots\circ h_2\circ h_1$, where each $h_i$ is equal to either $f$ or $g$.
Because $f[X]=g[X]=\textrm{EVEN}\cup\{a,b\}$ we have $h_1[X]=\textrm{EVEN}\cup\{a,b\}$. Because 
$$f[\textrm{EVEN}\cup\{a,b\}]=g[\textrm{EVEN}\cup\{a,b\}]=\{a,b\}$$
we have 
$$h_2\circ h_1[X]=\{a,b\}.$$
Because $f[\{a,b\}]=\{a\}$ and $g[\{a,b\}]=\{b\}$ we have for $n\geq3$
$$h_n\circ h_{n-1}\circ\ldots\circ h_2\circ h_1[X]=\{a\}$$
or
$$h_n\circ h_{n-1}\circ\ldots\circ h_2\circ h_1[X]=\{b\},$$
hence the IFS $\mathcal{F}=\{f,g\}$ is contractive.

Let $F$ be the Hutchinson operator induced by $\mathcal{F}$, namely $F:2^X\to2^X$ is defined as 
$$F(E):=f[E]\cup g[E].$$
We shall show that $F$ is not closed as a mapping from $2^X$ to $2^X$.

The one point set $\{a\}$ is not in the image $F[2^X]$ of the hyperspace $2^X$ via $F$. 
Thus it is enough to show that $\{a\}\in\overline{F[2^X]}$. Let $S(V_0;V_1,\ldots,V_k)$ be a base neigbourhood of $\{a\}$ which simply means here that $a\in\bigcap_{i=0}^kV_i$. Thus the set $\bigcap_{i=0}^kV_i$ is an open neighbourhood of $a$. Hence it must be a co-finite subset of $X$ and $2n\in\bigcap_{i=0}^kV_i$, for some $n\in\mathbb{N}$. This implies that $\{2n\}\in S(V_0;V_1,\ldots,V_k)$. We also have $\{2n\}=F(\{2n-1\})$. Hence $S(V_0;V_1,\ldots,V_k)\cap F[2^X]\neq\emptyset$. We conclude that $\{a\}\in \overline{F[2^X]}\setminus F[2^X]$.  

\end{proof}

We have shown that the Hutchinson operator of a contractive IFS on a compact $T_1$ space may not be closed and then it is not a topological contraction, because this is one of the conditions defining topological contraction. To apply directly Theorem A we assumed in \cite{MMRR} in the hypothesis of Theorem C that the Hutchinson operator was closed. It turns out, however, that to get the conclusion about the unique fixed point of the Hutchinson operator induced by a contractive IFS the assumption that the Hutchinson operator is closed is not necessary. Namely, we prove here the following stronger version of Theorem C.

\begin{thm}\label{attractor2}
If $X$ is a $T_1$ compact space then every contractive IFS has a unique attractor; in other words: the Hutchinson operator induced by this IFS has a unique fixed point.
\end{thm}

We present two proofs of this theorem. In both cases the existence of at least one fixed point of the Hutchinson operator induced by any IFS (\cite{MMRR}) is used  but the first proof is rather direct; the second and shorter one depends more heavily on some facts proven in \cite{MMRR}.

\vspace{0.4 cm}

\noindent {\it Proof 1.} Let $\mathcal{F}=\{f_1.\ldots,f_m\}$ be a contractive IFS. Let $F$ be the Hutchinson operator induced by $\mathcal{F}$. By Theorem 5 in \cite{MMRR} $F$ has a fixed point. Let us assume that $E_1,E_2\in2^X$ are two different fixed points of $F$. Let, e.g., $z\in E_1\setminus E_2$.
Let $U=E_2^c$ and $V=\{z\}^c$. We have $U\cup V=X$. Because $\mathcal{F}$ is a contractive IFS there exists $n\in\mathbb{N}$ such that for each sequence $(i_1,\ldots,i_n)\in\{1,\ldots,m\}^n$
$$f_{i_n}\circ\ldots\circ f_{i_1}[X]\subseteq U \;\;\textrm{or}\;\;f_{i_n}\circ\ldots\circ f_{i_1}[X]\subseteq V.$$
Because $E_1$ is a fixed point of $F$ we have $F^n[E_1]=E_1$. Hence there exists $s\in E_1$ and a sequence $(i_1,\ldots,i_n)\in\{1,\ldots,m\}^n$ such that 
$$f_{i_n}\circ\ldots\circ f_{i_1}(s)=z$$
and this implies
$$f_{i_n}\circ\ldots\circ f_{i_1}[X]\subseteq U.$$
Because 
$$f_{i_n}\circ\ldots\circ f_{i_1}[E_2]\subseteq f_{i_n}\circ\ldots\circ f_{i_1}[X]$$
we infer that
$$f_{i_n}\circ\ldots\circ f_{i_1}[E_2]\subseteq U.$$
Because $U\cap E_2=\emptyset$ we obtain 
$$f_{i_n}\circ\ldots\circ f_{i_1}[E_2]\setminus E_2\neq\emptyset$$.
Hence
$$F(E_2)\setminus E_2 =\bigcup_{(j_1,\ldots,j_n)\in\{1,\ldots,m\}^m}f_{j_n}\circ\ldots\circ f_{j_1}\setminus E_2\neq\emptyset,$$
which is an obvious contradiction with $E_2$ being a fixed point of $F$. This concludes the proof.

$\,$\hfill$\Box$

\vspace{0.4 cm}

\noindent{\it Proof 2.} By inspection of the proof of Theorem 4 in \cite{MMRR} (here Theorem B above) one can see that the assumption that the Hutchinson operator was closed was not used to prove that it had property $(\star)$.  

Let $F$ be the Hutchinson operator induced by our IFS $\mathcal{F}$. Let $E_1$ and $E_2$ be two different fixed points of $F$. 
We have $F^n(E_1)=E_1$ and $F^n(E_2)=E_2$, for each $n\in\mathbb{N}$. Hence 
\begin{equation}\label{inclusion}
\{E_1,E_2\}\subseteq F^n[2^X].
\end{equation}
The pair of sets $U=\{E_1\}^c$, $V=\{E_2\}^c$ is an open cover of $2^X$ and, because $F$ satisfies $(\star)$, for some $n$ either $F^n[2^X]\subseteq U$ or $F[2^X]\subseteq V$, but this contradicts (\ref{inclusion}).  

$\,$\hfill$\Box$

{}

\end{document}